\documentclass [12pt,reqno]{amsart}
\usepackage{amsmath}
\usepackage{amsthm}
\usepackage{dsfont}
\newtheorem{thm}{Theorem} 
\usepackage{amsfonts, amssymb}
\usepackage[utf8]{inputenc} 
\usepackage[english]{babel}
\usepackage{mathtools}
\usepackage{ mathrsfs }
\usepackage{enumerate}
\usepackage{ amssymb }
\usepackage{cite}

\newtheorem{theorem}{Theorem}

\newtheorem{Corollary}[theorem]{Corollary}

\newtheorem*{theorem*}{Theorem}

\newtheorem{lemma}[thm]{Lemma}
\newtheorem{proposition}[thm]{Proposition}

\theoremstyle{definition}
\newtheorem{remark}[thm]{Remark}

\theoremstyle{definition}

\theoremstyle{definition}
\newtheorem{definition}[thm]{Definition}

\newcommand{\SRA}{\operatorname{SRA}}
\newcommand{\DSE}{\operatorname{DSE}}
\newcommand{\LEG}{\operatorname{LEG}}
\newcommand{\LRB}{\operatorname{LRB}}

\newcommand{\CAT}{\operatorname{CAT}}
\newcommand{\diam}{\operatorname{diam}}

\numberwithin{equation}{section}

\def\R{{\mathbb R}}

\def\N{{\mathbb N}}

\def\wt{\widetilde}

\def\GG{{\Gamma}}

\def\<{\langle}
\def\>{\rangle}
\def \ee{{\epsilon}}

\def \gg {{\gamma}}
\def \aa {{\alpha}}
\def \bb {{\beta}}
\def \ll {{\lambda}}
\def \tt {{\theta}}

\begin{document}

\title{Sets with small angles in self-contracted curves}


\begin{abstract}
We study metric spaces with bounded rough angles. 
E. Le Donne, T. Rajala and E. Walsberg  implicitly used this notion to show that
infinite snowflakes can not be isometrically embedded into finite dimensional Banach spaces.   

We show that bounded non-rectifiable self-contracted curves contain metric subspaces with bounded rough angles. 
Which provides rectifiability of bounded self-contracted curves in a wide class of metric spaces including reversible 
$C^{\infty}$-Finsler manifolds, locally compact $\CAT(k)$-spaces with locally extendable geodesics and locally compact Busemann spaces with locally extendable geodesics.
 
We also extend the result on non embeddability of infinite snowflakes to this class of spaces.  
 




\end{abstract}

\author{Vladimir Zolotov}
\address[Vladimir Zolotov]{Steklov Institute of Mathematics, Russian Academy of Sciences, 27 Fontanka, 191023 St.Petersburg, Russia, University of Cologne, Albertus-Magnus-Platz, 50923 Köln, Germany and Mathematics and Mechanics Faculty, St. Petersburg State University, Universitetsky pr., 28, Stary Peterhof, 198504, Russia.}
\email[Vladimir Zolotov]{paranuel@mail.ru}

\maketitle

\section{Introduction}
\subsection{Self-contracted curves}
Let $X$ be a metric space and $I \subset \R$ an interval. A curve $\gg:I \rightarrow X$ is said to be \textit{self-contracted} if,  
$$d(\gg(t_2),\gg(t_3)) \le d(\gg(t_1),\gg(t_3))\text{, for every $t_1 \le t_2 \le t_3.$}$$
We remark that $\gg$ is not necessarily a continuous curve.

Self-contracted curves arise as 
\begin{itemize}
\item{orbits of gradient flows of convex and quasiconvex functions in Euclidean spaces, see \cite[Proposition 6.2]{DLS},}
\item{gradient curves of lower semi-continuous quasi-convex functions in $\CAT(0)$-spaces, see \cite[Proposition 4.6]{OhRect},}
\item{trajectories of of the subgradient dynamical system, see \cite[Proposition 6.4]{DLS},}
\item{trajectories orthogonal to a convex foliation, see \cite[Proposition 6.6]{DLS}.}
\end{itemize}

We refer to \cite{DLS} and \cite{DDDL} for the background information on self-contracted curves.

We remind that a curve $\gg:I \rightarrow X$ is called rectifiable if 
$$\sup\Big\{\sum_{i=1}^{k-1}d(\gg(t_i), \gg(t_i+1)) \Big\vert  t_1,\dots,t_{k} \in I , t_1 < \dots < t_{k} \Big\} < \infty.$$

It was shown that bounded self-contracted curves are rectifiable in following spaces:
\begin{itemize}
\item{Euclidean spaces, see \cite{DDDL}, \cite{LMV}}
\item{Riemannian manifolds, see \cite{DDDR},}
\item{finite-dimensional normed spaces, see \cite{Le, ST},}
\item{$\CAT(0)$-spaces with some additional properties, see \cite{OhRect}.}
\end{itemize}

We give the following extension of those results.
\begin{theorem}\label{NoCurvesI}
Let $\frac{1}{2} < \aa < 1$ and let $X$ be
an $\SRA(\aa)$-free metric space (see Definition \ref{defSRAfree}(\ref{SRAfree})) or a proper locally $\SRA(\aa)$-free  metric space (see Definition \ref{defSRAfree}(\ref{locSRAfree})). Then every bounded self-contracted curve $\gg:I \rightarrow X$ is rectifiable.
\end{theorem}

In particular we obtain that bounded self-contracted curves are rectifiable in 
\begin{itemize}
\item{complete reversible $C^{\infty}$-Finsler manifolds,}
\item{complete locally compact $\CAT(k)$-spaces with locally extendable geodesics,}
\item{complete locally compact  Busemann spaces with locally extendable geodesics.}
\end{itemize}
And more generally in complete locally compact geodesic metric spaces satisfying $\LEG$ condition (see Definition \ref{defLEG}) and 
$\LRB$ condition (see Definition \ref{defLRB}).

The main feature of the paper is a new approach to proving the rectifiability of bounded self-contracted curves. Instead of studying   
self-contracted curves in a particular space we study their properties which are independent from the ambient space.
We introduce a class of spaces which play the role of an obstacle.  We call  them $\SRA(\aa)$ spaces.
We show that every non-rectifiable self-contracted curve contains an arbitrary large $\SRA(\aa)$ subspace, see Theorem \ref{SRAinTC}. On the other hand we give Theorems \ref{weakEFq}, \ref{Glob}, \ref{GlobQ} which provide examples of spaces without big $\SRA(\aa)$-subspaces. 






\subsection{$\SRA(\aa)$ spaces and their presence in self-contracted curves}
\begin{definition}
Let $X$ be a metric space and $0 < \aa < 1$. We say that $X$ is an $\SRA(\aa)$ space or that $X$ satisfies $\SRA(\aa)$ condition if for every $x,z,y \in X$ we have 
$$d(x,y) \le \max\{d(x,z) + \aa d(z,y), \aa d(x,z) + d(z,y)\}.$$
\end{definition}
This condition says that all angles in $X$ are small in a rough sense.
Note that in case if $X$ is a subspace of Euclidean space. An inequality 
$$\measuredangle xzy > \arccos(-\aa) = \pi - \arccos(\aa)$$
implies
$$cos(\pi - \measuredangle xzy) > \aa,$$
and hence 
$$d(x,y) > \max\{d(x,z) + \aa d(z,y), \aa d(x,z) + d(z,y)\}.$$

We need the following discrete version of a self-contracting curve considered as a separate metric space. 
We also reverse the order of points to make further proofs look more natural.

\begin{definition}
Let $X=\{x_1,\dots,x_n\}$ be a metric space. We say that $X$ is \textit{discrete self-expanding} space or shorter $\DSE$ space if, 
$$d(x_i,x_j) \le d(x_i,x_k)\text{, for every $1 \le i \le j \le k \le n.$}$$
We define $L(X)$ and $D(X)$ by 
$$L(X) = \sum_{i = 1}^{i = n-1}d(x_i,x_{i+1})\text{ and } D(X) = d(x_1,x_n).$$
\end{definition}
\begin{remark}
It is easy to show (see Lemma \ref{2lemma}) that 
$$\diam(X) \le 2D(X) \le \diam(X),$$
where $\diam(X) = \max_{x,y \in X}\{d(x,y)\}$.
\end{remark}

\begin{theorem} \label{SRAinTC}
For every $\frac{1}{2} < \aa < 1$ and $k \in \N$ there exists $C = C(\aa, k) > 0$ such that for every $\DSE$ space $X$ such that $\frac{L(X)}{D(X)} > C$ there exists a $k$-point $\SRA(\aa)$ subspace $Y \subset X$.
\end{theorem}

Theorem \ref{SRAinTC} is inspired by the following proposition. Which is a key ingredient for showing that 
finite dimensional normed spaces do not contain infinite snowflakes, see \cite{BS}.

\begin{proposition}\cite[proof of Theorem 1.1]{BS}\label{SAinSN}
For every $0 < \aa < 1$ and every metric space $(Y,d)$ the $\aa$-snowflake space $(Y,d^\aa)$ satisfies 
$\SRA(\aa)$ condition.
\end{proposition}

\subsection{$\SRA(\aa)$-free spaces}
\begin{definition} \label{defSRAfree}
Let $X$ be a metric space and  $0 < \aa < 1$.
\begin{enumerate}
\item{ For $k \in \N$ we say that $X$ is \textit{free of $k$-point $\SRA(\aa)$ subspaces} if 
for every $k$-point $Y \subset X$ $Y$ does not satisfy $\SRA(\aa)$ condition. \label{kSRAfree}}
\item{ We say that $X$ is \textit{$\SRA(\aa)$-free} if 
it is free of $k$-point $\SRA(\aa)$ subspaces for some $k \in \N$.\label{SRAfree}}
\item{ We say that $X$ is \textit{locally $\SRA(\aa)$-free} if for every 
$x \in X$ there exist $r >0$ such that $B_r(x)$ is $\SRA(\aa)$-free (considered as metric space with metrics induced from $X$).
\label{locSRAfree}} 
\end{enumerate}
\end{definition}

To prove that certain spaces are locally $\SRA(\aa)$-free we introduce two additional conditions: $\LEG$ condition and $\LRB$ condition.
Theorem \ref{weakEFq} provides that $\LEG$ condition and $\LRB$ condition imply local $\SRA(\aa)$-free condition for doubling
metric spaces.

\begin{definition}\label{defLEG}
Let $X$ be a  complete geodesic  metric space. We say that $x \in X$ satisfies \textit{locally extendable geodesics} condition or shorter $\LEG$ condition with a parameter  $\ee > 0$ if   for every unit speed
 minimizing geodesics  $\gg: [0,T] \rightarrow  B_\ee(x)$ there exist a unit speed minimizing geodesic $\wt \gg: [0,T + \ee] \rightarrow X$
extending it i.e. such that,
$$\wt \gg(t) = \gg(t)\textit{, for every $t \in [0,T]$}.$$
We say that $X$ satisfies $\LEG$ condition if every $x \in X$ satisfies $\LEG$ condition with some $\ee(x) > 0$.
\end{definition}

\begin{definition}\label{defLRB}
Let $X$ be a complete geodesic metric space. We say that $x \in X$ satisfies \textit{locally rough Busemann} condition or shorter $\LRB$ condition with parameters  $H, K > 0$ if   for  every pair of 
unit speed minimizing geodesic  $\gg_1: [0,L_1] \rightarrow  B_H(x)$, $\gg_2: [0,L_2] \rightarrow  B_H(x)$ s.t.,
$$\gg_1(0) = \gg_2(0),$$
we have 
$$d(\gg_1(tL_1),\gg_2(tL_2)) \le Ktd(\gg_1(L_1),\gg_2(L_2))\text{, for every $t\in[0,1]$.}$$
We say that $X$ satisfies $\LRB$ condition if every $x \in X$ satisfies $\LRB$ condition with some $H(x), K(x) > 0$.
\end{definition}
The class of spaces satisfying $\LRB$ condition includes following subclasses: 
\begin{enumerate}
\item{complete reversible $C^{\infty}$-Finsler manifolds,}
\item{locally $L$-convex spaces, see \cite{OhtaConv}. This subclass is huge by itself and contains all next subclasses, see \cite[Section 3]{OhtaConv},}
\item{Alexandrov spaces with a local upper curvature bound. The subclass of of locally compact
spaces with an upper curvature bound and extendable geodesics is studied
 in \cite{LN}.  }
\item{Strictly convex Banach spaces.}   
\end{enumerate}

\begin{remark}
Banach spaces which are not strictly convex do not satisfy $\LRB$ condition as it formulated.  
In order to make our approach work for this case one should modify our definitions in the style of the bicombings theory. 
This means that one fixes a set of ``good'' minimizing geodesics in the metric space such that every two points can be connected 
by a ``good'' minimizing geodesic. In the case of Banach spaces we can take  straight lines as such a set. In this setting we say that $\LEG$-condition 
is satisfied if any ``good'' minimizing geodesic can be extended and to a longer ``good'' minimizing geodesic. And $LRB$ condition is required to hold 
only for ``good'' geodesics. In this setting Theorems \ref{weakEFq}, \ref{EmbDoub} hold 
and have basically have the same proofs.   
\end{remark}

\begin{theorem}\label{weakEFq}
Let $X$ be a geodesic space and $x \in X$ be a point satisfying $\LEG$ condition with $\ee > 0$ and $LRB$ with constants 
$H > 0, K \ge 1$. Let $R = \frac{1}{2}\min\{\ee, H\}$ and suppose that $B_R(x)$ has doubling property with a constant $\ll \in \N$. Then for every $0 < \aa < 1$ $B_{\frac{R(1- \aa)}{6K}}(x)$ is free of $N$-point $SRA(\aa)$ subspaces,
where $N \in \N$ depends only on $K, \aa, \ll$.  
\end{theorem}


Previous theorem can be considered as generalization of the following proposition see, \cite{EF,KS}.
\begin{proposition}\label{propEF}
For any $n \in \N$ and $0 < \aa < \pi$ there exists $N \in \N$ such that if $S \subset \R^n$
has cardinality at least $N$ then there are distinct $x,z,y \in S$ such that $\aa \le \measuredangle xzy \le \pi$.
\end{proposition}

In Theorem \ref{weakEFq} instead of dealing with locally compact case we consider a more restrictive local doubling condition.
However for spaces satisfying $\LEG$ and $\LRB$ this is the same as the following theorem says.



\begin{theorem}\label{EmbDoub}
Let $X$ be a locally compact geodesic metric space and $x \in X$ be a point satisfying $\LEG$ condition with a constant $\ee > 0$ and 
$\LRB$ condition with constants $H > 0, K \ge 1$. Then for every $r > \frac{\min\{\ee, H\}}{12K}$  
\begin{enumerate}
\item{there is a bi-Lipschitz embedding of $B_r(x)$ into a finite dimensional Euclidean space,\label{Emb}} 
\item{$B_r(x)$ satisfies doubling condition (considered as metric space with metrics induced from $X$).\label{Doub}}
\end{enumerate}
\end{theorem}
 
Theorem \ref{EmbDoub}(\ref{Emb}) can be considered as a generalization of \cite[Theorem 1.1]{LP}.





\subsection{Local to global for $\SRA(\aa)$-free spaces}

\begin{theorem}\label{Glob}
Let $0 < \aa < 1$ and $X$ be a compact locally $\SRA(\aa)$-free metric space. Then $X$ is $\SRA(\aa)$-free.  
\end{theorem}

In particular we have that compact subsets of locally $\SRA(\aa)$-free spaces are $\SRA(\aa)$-free (considered as metric spaces with metrics induced from  ambient spaces).

\begin{theorem}\label{GlobQ}
Let $X$ be a metric space, $x \in X$, $0 < \aa < 1$, $0 < r < R$, $k \in \N$. Suppose that $B_R(x)$ satisfies doubling condition with a constant $\ll \in \N$ and for every $y \in B_R(x)$ $B_r(x)$ is free of $k$-point $SRA(\aa)$ subspaces.  Then  
$B_R(x)$ is free of $k\ll^{\left \lceil{\log_2(\frac{R}{r})}\right \rceil}$-point $\SRA(\aa)$ subspaces, where
$\left \lceil{\cdot}\right \rceil $  denotes the rounded up integer part.
\end{theorem}

\subsection{Non embeddability of snowflakes}
As a direct implication of Proposition \ref{SAinSN} we have the following.
\begin{Corollary}
For $0 < \aa < 1$ an $\aa$-snowflake of a metric space containing at least $N$ points does not embed isometrically 
into any metric space which is free of $N$-point $\SRA(\aa)$ subspaces.
\end{Corollary}

\section{Notation}
Let $X$ be a metric space. $X$ is said to be \textit{geodesic space} if for every pair of points there exists a continuous curve 
$\gg: [0,1] \rightarrow X$ such that $$d(x,y) = L(\gg) \overset{def}{=} 
\sup_{n \in \N} \sup_{0 \le t_1 < \dots < t_n \le 1}\sum_{i=1}^{n-1}d(\gg(t_i),\gg(t_{i+1})).$$

For a metric space $X$, $x \in X$ and $R > 0$ we denote by $B_R(x)$, $\overline{B_R(x)}$ and $S_R(x)$
an open ball, a closed ball and a sphere with radius $R$ and center $x$ respectively. 
 
For a metric space $(Y,d)$ an $0 < \bb < 1$ we denote by $(Y,d^\bb)$ a $\bb$-snowflake of $(Y,d)$ which metric is given by 
$$d_{(Y,d^\bb)}(y_1,y_2) = \Big (d_{(Y,d)}(y_1,y_2)\Big )^\bb.$$

\section{Proof of Theorem \ref{SRAinTC}}
\begin{lemma}\label{2lemma}
Let $X = \{x_1,\dots,x_n\}$ be a $\DSE$ space. Suppose that $1 \le i \le j \le k \le l \le n$ then 
$$d(x_j, x_k) \le 2d(x_i, x_l).$$
\end{lemma}
\begin{proof}
The proof is a straightforward chain of inequalities based on the definition of  a $\DSE$ space and the triangle inequality,
$$d(x_j, x_k) \le d(x_j, x_l) \le d(x_i, x_j) + d(x_i, x_l) \le 2d(x_i, x_l).$$
\end{proof}

Consider a $\DSE$ space $X = \{x_1,\dots, x_n\}$. For every $1 \le i < j < k \le n$
$$d(x_i,x_j) \le d(x_i,x_k) \le \max\{d(x_i,x_k) + \aa d(x_j,x_k), \aa d(x_i,x_k) + d(x_j,x_k)\}.$$
Thus, the following two condition imply $\SRA(\aa)$ condition for $X$
\begin{enumerate}
\item{$d(x_i,x_k) \le d(x_i,x_j) + \aa d(x_j,x_k)$, for every $1 \le i < j < k \le n$, } \label{c1}
\item{$d(x_j,x_k) \le d(x_i,x_k) + \aa d(x_i,x_j)$, for every $1 \le i < j < k \le n$, } \label{c2}
\end{enumerate}   
Our proof of Theorem \ref{SRAinTC} consists of two steps. First we find a subset of $\DSE$ space satisfying (\ref{c1})
by applying Lemma \ref{StandAngles}. Then by Lemma \ref{WeirdAngles} we find a subset of this new set which satisfies both (\ref{c1}) and (\ref{c2}).

\begin{lemma}\label{StandAngles}
Suppose that $0 < \tt < 1$, $m \in \N$ and  $X = \{x_1,\dots,x_n\}$ is a $\DSE$ space
 with $$L(X) \ge C(m, \tt) D(X),$$ 
 where  $C(m, \tt) =\big(\frac{  (m(m-1))^{m-1} }{\tt^{m-2}} + 2m\big)$.
Then there exist $1 \le a_1 < \dots < a_m \le n$ such that for every $1 \le i < j < k \le m$ we have $$d(x_{a_i},x_{a_k}) \le d(x_{a_i},x_{a_j}) + \tt d(x_{a_j},x_{a_k}).$$
\end{lemma} 
\begin{proof}
By contradiction we suppose that for every $1 \le a_1 < \dots < a_m \le n$ there exist $i < j < k$ such that
\begin{equation} d(x_{a_i},x_{a_k}) > d(x_{a_i},x_{a_j}) + \tt d(x_{a_j},x_{a_k}). \label{eq} \end{equation}
We define inductively a sequence of $m$-point subsets $P^1,\dots,P^T \subset \{1,\dots,n\}$.
We define $P^1$ by 
$$P^1 = \{1,\dots,m\}.$$
For every $t = 1,\dots,T$ 	we are going to use the following notation for elements of $P^t$,
$$P^t = \{{p_1^t},\dots,{p_m^t}\},$$
where ${p_1^t} <\dots < {p_m^t}$. 
Suppose that we already defined $P^t$ and now we have to define $P^{t+1}$.
By (\ref{eq}) we can fix $i(t), j(t), k(t) \in P^t$ such that $i(t) < j(t) < k(t)$ and 
\begin{equation} d(x_{ i(t)},x_{ k(t)}) > d(x_{ i(t)},x_{ j(t)}) + \tt d(x_{ j(t)},x_{ k(t)}). \label{eqq} \end{equation}
We denote by $d(t)$ the size of the set $P^t \cap [j(t),k(t))$. In case if $d + p_m^t > n$ we stop the process and say that $T = t$.
Otherwise we define $P^{t+1}$ by 
$$P^{t+1} = \big( P^t \setminus [j(t),k(t)) \big) \cup \{p_m^t +1,\dots,p_m^t +d(t)\}.$$
Now when we finished defining the sequence $P^1,\dots,P^{T}$ we are going to define weights $c_1,\dots,c_{m-1} > 0$.
We take $c_{m-1} = 1$. Other weights $c_1,\dots,c_{m-2}$ should be such that $c_u$ is much bigger then $c_{u+1}$ for every $u = 1,\dots,m-2$.
More precisely we take
$$c_u = \frac{m(m-1)}{\tt}c_{u+1}.$$
So a direct formula is,
$$c_u = \Big(\frac{m(m-1)}{\tt}\Big)^{m-1-u}.$$

For $t = 1,\dots,T$ we define $S^t$ by
$$S^t = \sum_{1 \le a < b \le m}c_ad(x_{p^t_a}, x_{p^t_b}).$$  
We claim that for every $t = 1,\dots,T-1$ we have 
\begin{equation} S^t \le  S^{t+1} - \sum_{u = p_m^t}^{u= p_m^{t+1} -1 }d(x_{u},x_{u+1}). \label{CL} \end{equation}
\begin{proof}
We chop $S^t$ into $3$-pieces in order to bound them separately 
$$S^t =\sum_{\substack{{1 \le a < b \le m}\\
                  p^t_a < j(t)}}
         c_ad(x_{p^t_a}, x_{p^t_b})
          + \sum_{\substack{{1 \le a < b \le m}\\
                  p^t_a = j(t)}}
         c_ad(x_{p^t_a}, x_{p^t_b}) + 
\sum_{\substack{{1 \le a < b \le m}\\
                  p^t_a > j(t)}}
         c_ad(x_{p^t_a}, x_{p^t_b}).$$
         
In the case $p^t_a < j(t)$ we have that $p^t_a = p^{t+1}_a$ and $p^t_b \le p^{t+1}_b$. This implies, 
$$d(x_{p^t_a}, x_{p^t_b}) \le d(x_{p^{t+1}_a}, x_{p^{t+1}_b}).$$
Applying (\ref{eqq}) we obtain 
$$d(x_{p^t_a}, x_{p^t_b}) \le d(x_{p^{t+1}_a}, x_{p^{t+1}_b}) - \tt d(x_{j(t)}, x_{k(t)}),$$ 
for the special case $p^t_a = i(t)$ and $p^t_b = j(t)$.
Summing all those inequalities provides 
\begin{equation}\sum_{\substack{{1 \le a < b \le m}\\
                  p^t_a < j(t)}}
         c_ad(x_{p^t_a}, x_{p^t_b}) \le 
         \sum_{\substack{{1 \le a < b \le m}\\
                  p^t_a < j(t)}}
         c_ad(x_{p^{t+1}_a}, x_{p^{t+1}_b})  - \tt c_{\hat a}d(x_{j(t)}, x_{k(t)}),\label{1ndPF}\end{equation}
         where $\hat a$ is such that $p^t_{\hat a} = i(t)$.

Now we are going to provide a bound for 
$$\sum_{\substack{{1 \le a < b \le m}\\
                  p^t_a = j(t)}}
         c_ad(x_{p^t_a}, x_{p^t_b}).$$

In the case  $p^t_b \le k(t)$ from the definition of $\DSE$ space we have 
$$d(x_{p^t_a}, x_{p^t_b}) \le d(x_{j(t)},x_{k(t)}).$$
In the case $p^t_b \ge k(t)$ from the triangle inequality for $\Delta x_{p^t_a} x_{k(t)} x_{p^t_b}$ we have 
$$d(x_{p^t_a}, x_{p^t_b}) \le d(x_{j(t)},x_{k(t)}) + d(x_{k(t)}, x_{p^t_b}).$$
Summing those inequalities provides 
\begin{equation}\sum_{\substack{{1 \le a < b \le m}\\
                  p^t_a = j(t)}}
         c_ad(x_{p^t_a}, x_{p^t_b}) \le 
         \Big(\sum_{\substack{{1 \le a < b \le m}\\
                  p^t_a = j(t)}}
         c_a\Big)d(x_{j(t)},x_{k(t)}) + 
         \Big(\sum_{\substack{{p^t_a = j(t)}\\
                  p^t_b > k(t)}}
         c_a\Big)d(x_{k(t)},x_{p^t_b}).\label{2ndP}\end{equation}
Note that if $p^t_a = j(t)$ then $p^{t+1}_a = k(t)$ and also note that 
$$\{ u \in P^{t+1}| u > k(t)\} \supseteq \big(\{ u \in P^{t}| u > k(t)\} \cup \{p^t_m+1\}\big).$$
Hence,
$$\Big(\sum_{\substack{{p^t_a = j(t)}\\
                  p^t_b > k(t)}}
         c_a\Big)d(x_{k(t)},x_{p^t_b}) \le 
         \sum_{\substack{{1 \le a < b \le m}\\
                  p^t_a = j(t)}}
         c_ad(x_{p^{t+1}_a}, x_{p^{t+1}_b}) - c_{\wt a}d(x_{k(t)},x_{p^t_m + 1}),
         $$
         where $1 \le \wt a \le m$ is such that  $p^t_{\wt a} = j(t)$.
Substituting the last inequality into (\ref{2ndP}) provides
\begin{equation}\sum_{\substack{{1 \le a < b \le m}\\
                  p^t_a = j(t)}}
         c_ad(x_{p^t_a}, x_{p^t_b}) \le 
         \Big(\sum_{\substack{{1 \le a < b \le m}\\
                  p^t_a = j(t)}}
         c_a\Big)d(x_{j(t)},x_{k(t)}) + $$$$+
         \sum_{\substack{{1 \le a < b \le m}\\
                  p^t_a = j(t)}}
         c_ad(x_{p^{t+1}_a}, x_{p^{t+1}_b}) - c_{\wt a}d(x_{k(t)},x_{p^t_m + 1}).\label{2ndPF}\end{equation}
        
         Finally we are going to provide a bound for $$\sum_{\substack{{1 \le a < b \le m}\\
                  p^t_a > j(t)}}
         c_ad(x_{p^t_a}, x_{p^t_b}).$$         
         If $p_a^t > j(t)$ then by the triangle inequality and Lemma \ref{2lemma} we have,
        $$d(x_{p^t_a}, x_{p^t_b}) \le d(x_{p^t_a},x_{k(t)}) + d(x_{k(t)},  x_{p^t_b}) \le 2d(x_{j(t)},x_{k(t)}) +  d(x_{k(t)},  x_{p^t_m}).$$
        In the case  $p_b^t \le k(t)$ this can be improved to 
         $$d(x_{p^t_a}, x_{p^t_b})   \le 2d(x_{j(t)},x_{k(t)}). $$      
        We conclude that 
        \begin{equation}\sum_{\substack{{1 \le a < b \le m}\\
                  p^t_a > j(t)}}
         c_ad(x_{p^t_a}, x_{p^t_b}) \le 
         2d(x_{j(t)},x_{k(t)})\sum_{\substack{{1 \le a < b \le m}\\
                  p^t_a > j(t)}}
         c_a +
          d(x_{k(t)},  x_{p^t_m})\sum_{\substack{{1 \le a < b \le m}\\
                  p^t_a > j(t),
                  p^t_b > k(t)
                  }}
         c_a.\label{3ndPF} \end{equation}
         
        Now we can bound $S^t$ by summing up  (\ref{1ndPF}), (\ref{2ndPF}) and (\ref{3ndPF}),
         \begin{equation}S^t \le  
         \sum_{\substack{{1 \le a < b \le m}\\
                  p^t_a < j(t)}}
         c_ad(x_{p^{t+1}_a}, x_{p^{t+1}_b}) +
         \sum_{\substack{{1 \le a < b \le m}\\
                  p^t_a = j(t)}}
         c_ad(x_{p^{t+1}_a}, x_{p^{t+1}_b}) - R_1  - R_2,\label{summV1}\end{equation} where
         $$ R_1 = 
         d(x_{j(t)},x_{k(t)})(\tt c_{\hat a} - \sum_{\substack{{1 \le a < b \le m}\\
                  p^t_a = j(t)}}
         c_a - 2 \sum_{\substack{{1 \le a < b \le m}\\
                  p^t_a > j(t)}}
         c_a),$$
         $$R_2 =   (c_{\wt a}d(x_{k(t)},x_{p^t_m + 1}) - d(x_{k(t)},  x_{p^t_m})\sum_{\substack{{1 \le a < b \le m}\\
                  p^t_a > j(t),
                  p^t_b > k(t)}}
         c_a),$$
         $$\text{$1 \le \hat a, \wt a \le m$ are such that $p^t_{\hat a} = i(t)$, $p^t_{\wt a} = j(t)$}.$$
We can rewrite (\ref{summV1}) as
$$S^t \le  S^{t+1}
          - \sum_{\substack{{1 \le a < b \le m}\\
                  p^t_a > j(t)}}
         c_ad(x_{p^{t+1}_a}, x_{p^{t+1}_b})  - R_1  - R_2.$$
         
Since $\hat a < a$ for every $1 \le a \le n$ such that $p^t_a \ge j(t)$. Thus $c_{\hat a} \ge \frac{m(m-1)}{\tt}c_a$ and $R_1 \ge 0$.
Thus, 
$$S^t \le  S^{t+1}
          - \sum_{\substack{{1 \le a < b \le m}\\
                  p^t_a > j(t)}}
         c_ad(x_{p^{t+1}_a}, x_{p^{t+1}_b}) - R_2.$$
Now there are two cases:

\textit{Case A:  $k(t) < p^t_m$.} Then, 
$$\sum_{\substack{{1 \le a < b \le m}\\
                  p^t_a > j(t)}}
         c_ad(x_{p^{t+1}_a}, x_{p^{t+1}_b}) \ge \sum_{u = p_m^t}^{u= p_m^{t+1} -1 }d(x_{u},x_{u+1}).$$
To finish the proof of the claim in this case it suffices to show that,
$$R_2 \ge 0.$$

Since $\wt a < a$ for every $1 \le a \le n$ such that $p^t_a > j(t)$. Thus $c_{\wt a}\ge \frac{m(m-1)}{\tt}c_a \ge m(m-1)c_a$ and $R_2 \ge 0$.

\textit{Case B:  $k(t) = p^t_m$.}
In this case 
$$R_2 = c_{\wt a}d(x_{k(t)},x_{p^t_m + 1}) \ge d(x_{p^t_m},x_{p^t_m + 1}).$$
We also have 
$$\sum_{\substack{{1 \le a < b \le m}\\
                  p^t_a > j(t)}}
         c_ad(x_{p^{t+1}_a}, x_{p^{t+1}_b})  \ge \sum_{u = p_m^t + 1}^{u= p_m^{t+1} -1 }d(x_{u},x_{u+1}),$$
and the claim follows. 
\end{proof}
We sum inequalities (\ref{CL}) for $t = 1, \dots, T - 1$ and get, 
$$S^1 \le S^T -  \sum_{u = m}^{u= p_m^{T} -1 }d(x_{u},x_{u+1}).$$
Which implies 
$$\sum_{u = 1}^{u=m -1 }d(x_{u},x_{u+1}) \le S^T -  \sum_{u = m}^{u= n -1 }d(x_{u},x_{u+1}) +  \sum_{u= p_m^{T}}^{u = n-1}d(x_{u},x_{u+1}).$$
Thus, 
\begin{equation}L(X) \le S^T +  \sum_{u= p_m^{T}}^{u = n-1}d(x_{u},x_{u+1}) = \sum_{1 \le a < b \le m}c_ad(x_{p^T_a}, x_{p^T_b}) + \sum_{u= p_m^{T}}^{u = n-1}d(x_{u},x_{u+1}).\label{dc}\end{equation}
By Lemma \ref{2lemma} we have for every $1 \le u < v \le n$,
$$d(x_u,x_v) \le 2d(x_1,x_n) = 2D(X).$$
We apply the last inequality to  (\ref{dc}), which provides
$$L(X) < (m(m-1)c_1 + 2m) D(X) = C(m,\tt)D(X).$$
Thus, we have a contradiction with the assumption of the lemma.
\end{proof}

\begin{lemma}\label{WeirdAngles}
Let $0 < \tt < 1$ and $\aa  > \frac{1}{2}\frac{1+\tt}{1-\tt}$ then there exists $n = n(\tt,\aa)$ such that 
there is no $Z = \{z_1,\dots,z_n\}$ $\DSE$ space such that 
\begin{enumerate}
\item{$d(z_{i},z_{k}) \le d(z_{i},z_{j}) + \tt d(z_{j},z_{k})$, for every $1 \le i < j < k \le n$. \label{c1}}
\item{$d(z_n,z_{i+1}) \ge d(z_n,z_{i}) + \aa d(z_{i},z_{i+1})$, for every $1 \le i \le n - 1$. \label{c2}}
\end{enumerate}
\end{lemma}
\begin{proof}
Suppose that $Z$ is DSE space satisfying (\ref{c1}) and (\ref{c2}).
From (\ref{c2}) we have 
$$d(x_1,x_n) \le d(x_2,x_n) \le \dots \le d(x_{n-1},x_n).$$
Thus,
$$d(x_{n-1},x_n) - d(x_1,x_n) \ge 0.$$
We use the following notation 
$$L = d(x_1,x_n)\text{ and }a = d(x_{n-1},x_n) - d(x_1,x_n).$$
\textit{Claim:} for every $1 \le i \le n -1$ we have 
\begin{equation}d(x_i,x_{i+1}) \ge \Big(\frac{1 - \tt}{2}\Big)^{n-1-i}(L+a).\label{pf}\end{equation}

We prove the claim by induction. The base $i = n - 1$ is trivial. Suppose that the claim is proven for $i = k+1$ and we have to 
 validate it for $i = k$. 
From $(\ref{c1})$ we have $$d(x_k ,x_{k+2}) \le d(x_k, x_{k+1}) + \tt d(x_{k+1},x_{k+2}).$$ 
Triangle inequality for $\Delta x_k x_{k+1} x_{k+2}$ provides
$$d(x_{k+1},x_{k+2}) -  d(x_{k},x_{k+1}) \le  d(x_k, x_{k+2}).$$
Combining those inequalities we obtain 
$$d(x_{k+1},x_{k+2}) -  d(x_{k},x_{k+1}) \le d(x_k, x_{k+1}) + \tt d(x_{k+1},x_{k+2}).$$ 
Thus, 
$$d(x_k, x_{k+1}) \ge \Big(\frac{1-\tt}{2}\Big) d(x_{k+1},x_{k+2}).$$ 
This provides the claim.

From (\ref{c2}) we have 
$$d(x_n, x_1) + \aa d(x_1, x_2) \le d(x_2, x_n),$$ 
$$d(x_n, x_2) + \aa d(x_2, x_3) \le d(x_3, x_n),$$ 
$$\dots$$
$$d(x_n, x_{n-2}) + \aa d(x_{n-2}, x_{n-1}) \le d(x_{n-1}, x_n).$$ 

We sum these inequalities and get
$$d(x_n,x_1) + \aa (d(x_1, x_2) + d(x_2, x_3) + \dots + d(x_{n-2}, x_{n-1})) \le d(x_{n-1}, x_n). $$
Substituting (\ref{pf}) provides 
$$L + \aa  \Big(\frac{1 - \tt}{2}\Big) \Big( \Big(\frac{1 - \tt}{2}\Big)^{n-2} + \dots + 1\Big)(L + a) \le L + a.$$ 
The last inequality is equivalent to 
$$L \le (L+a)\Big(1 - \aa \Big(\frac{1 - \tt}{2}\Big)\Big(1 + \dots + \Big(\frac{1 - \tt}{2}\Big)^{n-2} \Big)\Big).$$ 
By Lemma \ref{2lemma} we have
$a \le L.$
Thus,
$$1 \le 2\Big(1 - \aa \Big(\frac{1 - \tt}{2}\Big)\Big(1 + \dots + \Big(\frac{1 - \tt}{2}\Big)^{n-2} \Big)\Big).$$ 
Which is equivalent to 
$$\aa \le \frac{1}{2\Big(\frac{1 - \tt}{2}\Big)\Big(1 + \dots + \Big(\frac{1 - \tt}{2}\Big)^{n-2} \Big)}.$$
Note that $\aa  > \frac{1}{2}\frac{1+\tt}{1-\tt}$ and 
$$\frac{1}{2\Big(\frac{1 - \tt}{2}\Big)\Big(1 + \dots + \Big(\frac{1 - \tt}{2}\Big)^{m-2} \Big)} \underset{m \rightarrow \infty}{\longrightarrow} \frac{1}{2}\frac{1+\tt}{1-\tt}.$$
We conclude that there exists $N \in \N$ such that
$$\aa > \frac{1}{2\Big(\frac{1 - \tt}{2}\Big)\Big(1 + \dots + \Big(\frac{1 - \tt}{2}\Big)^{N-2} \Big)}.$$
And thus $n$ have to be less then $N$.
\end{proof}

\begin{proof}[Proof of Theorem \ref{SRAinTC}] 
There exists $0 < \tt < \frac{1}{2}$ such that $\aa > \frac{1}{2}\frac{1+\tt}{1-\tt}$.  Let $n(\tt,\aa) \in \N$ be a number provided 
by Lemma \ref{WeirdAngles}.

By Ramsey's theorem there exists $m \in \N$ such that if the set of all $3$-point subsets of $\{1,\dots,m\}$ are colored in two colors say red and blue, then either there is a $k$-point subset of $\{1,\dots,m\}$ such that all its $3$-point subsets are red or there is 
a $n(\tt,\aa)$-point subset of $\{1,\dots,m\}$ such that all its $3$-point subsets are blue.

We denote by $C = C(m, \tt) > 0$ the constant provided by Lemma \ref{StandAngles}. 

Suppose that $X$ is a $\DSE$ space such that $\frac{L(X)}{D(X)} > C$. By Lemma \ref{StandAngles} we have an $m$-point 
subspace $Y = \{y_1,\dots,y_m\} \subset X$ s.t., for every $1 \le i < j < k \le m$,
$$d(y_i,y_k) \le d(y_i,y_j) + \tt d(y_j,y_k).$$

We color the $3$-point subsets of $\{1,\dots,m\}$ in two colors. For $i < j < k$ we say that $\{i,j,k\}$ is red if 
$$d(y_j,y_k) \le d(y_i, y_k) + \aa d(y_i,y_j),$$
otherwise we say that $\{i,j,k\}$ is blue.

By definition of $m$ we have that one of the following holds.

\textit{Case A}: there exists an $k$-point subset $I \subset \{1,\dots,m\}$  such that all  its $3$-point subsets are red. In this case $\{y_i\}_{i \in I}$ is that required $k$-point subset satisfying $\SRA(\aa)$.

\textit{Case B:} there exists an $n(\tt,\aa)$-point subset $\{i_1,\dots,i_n\} \subset  \{1, \dots ,m\}$ such that all its $3$-point subsets are blue. For $1 \le l \le n$ we denote by $z_l$ a point $y_{i_l}$. Note that $Z = \{z_1,\dots,z_n\}$ satisfies, 
$$d(z_{i},z_{k}) \le d(z_{i},z_{j}) + \tt d(z_{j},z_{k})\text{, for every }1 \le i < j < k \le n,$$ 
$$d(z_n,z_{i+1}) \ge d(z_n,z_{i}) + \aa d(z_{i},z_{i+1})\text{, for every }1 \le i \le n - 1,$$
which contradicts Lemma \ref{WeirdAngles}.
\end{proof}

\section{Proofs of Theorems \ref{weakEFq}, \ref{EmbDoub}}


\begin{lemma}\label{FarOnSphere}
Let $X$ be a geodesic metric space and $x \in X$ be a point satisfying $\LRB$ condition with parameters $H > 0, K \ge 1$. Let $0 < r < R \le \frac{H}{2}$.
Suppose that  $p, y, y' \in X$, $z, z' \in B_R(x)$ are such that  
$d(p, x) < r,$ 
$d(p, y) = d(p, y'), d(p, z) = d(p, y) + d(y, z), d(p, z') = d(p, y') + d(y', z')$ 
and 
\begin{equation}d(y,y') > \bb d(p, y),\label{wwwNew}\end{equation} for some $0 < \bb < 1$.
Then, 
\begin{enumerate}
\item{$d(z,z') > \frac{\bb(R - r)}{K} - 2r,$}\label{FOS1}
\item{in particular if $r < \frac{R\bb}{6K}$ then $d(z,z') > 3r.$}\label{FOS2}
\end{enumerate}
\end{lemma}
\begin{proof}
From the triangle inequality for $\Delta x p z$ we have 
\begin{equation}R - r \le d(p, z) \le R + r.\label{e1}\end{equation}
The same argument provides
\begin{equation}R - r \le d(p, z') \le R + r.\label{e2}\end{equation}
Without loss of generality we can assume that $d(p, z) \le d(p, z')$.
Let $\wt z'$ be a point on a geodesic between $p$ and $z'$ such that $d(p,\wt z') = d(p, z)$. 
From (\ref{e1}) and (\ref{e2}) we have 
\begin{equation}d(\wt z', z') \le 2r\label{e3}.\end{equation}

$\LRB$-condition implies that
$$d(z, \wt z') \ge \frac{1}{K}\frac{d(p, z)}{d(p, y)}d(y, y') \overset{(\ref{wwwNew})}{>} $$ 
$$\ge \frac{1}{K}\frac{d(p, z)}{d(p, y)}\bb d(p, y) = \frac{\bb d(p, z)}{K} \overset{(\ref{e1})}{\ge}  \frac{\bb(R - r)}{K}.$$
The last inequality and (\ref{e3}) provide (\ref{FOS1}). And (\ref{FOS2}) follows directly from (\ref{FOS1}).
\end{proof}

\begin{proof}[Proof of Theorem \ref{EmbDoub}]
Note that Theorem \ref{EmbDoub}(\ref{Doub}) follows directly from Theorem \ref{EmbDoub}(\ref{Emb}).  In the following we prove Theorem \ref{EmbDoub}(\ref{Emb}).
We take 
$R = \frac{\min\{\ee, H\}}{2}$ and also fix $0 < \gg < 1$ such that $r < \frac{R(1- \gg)}{6K}$.
Since $X$ is geodesic and locally compact we have that  $S_{R}(x)$   is compact, see \cite[Proposition 2.5.22]{BBI}. 
Thus there exists a finite $r$-net $z_1, \dots, z_N$ in $S_{R}(x)$.

Let $\Phi:B_R(x) \rightarrow \R^N$ be a map  given by 
$$\Phi(x) = (d(x, z_1), \dots, d(x, z_N)).$$
We claim that it can be taken as the required embedding.
Obviously $\Phi$ is Lipschitz.  In the remaining part of the proof we show that for every $p, y \in B_r(x)$,
$$\vert \vert \Phi(p) - \Phi(y) \vert \vert \ge \gg d(p,y).$$
To do this we prove that for every $p,y \in B_r(x)$ there exist $i = 1,\dots,N$ such that 
$$\vert d(p, z_i) - d(y, z_i) \vert \ge \gg d(p,y).$$
By contradiction suppose that for every $i = 1,\dots,N$ 
\begin{equation}\vert d(p, z_i) - d(y, z_i) \vert \le \gg d(p,y).\label{ssss1}\end{equation}
By $\LEG$ property there exists $w \in S_{R}(x)$ such that $d(p, w) = d(w, y) + d(y, w)$.  
We claim that $d(z_i, w) > 3r$ for every $i = 1,\dots,N$ which contradicts the assumption that $z_1, \dots, z_N$ is a $r$-net.
Let $y' \in X$ be such that 
\begin{equation}d(p, y') = d(p, y)\label{qddef0}.\end{equation} and 
\begin{equation}d(p,y') + d(y',z_i) = d(p, z_i)\label{qddef}.\end{equation}
By the triangle inequality we have 
$$d(y,y') \ge d(z_i,y) - d(z_i,y') \overset{(\ref{ssss1})}{\ge}$$ 
$$\ge d(z_i,p) - \gg d(p, y) - d(z_i,y') \overset{(\ref{qddef})}{=}$$  
$$= d(p, y') - \gg d(p, y) \overset{(\ref{qddef0})}{=}  (1 - \gg) d(p, y).$$
Thus by Lemma \ref{FarOnSphere}(\ref{FOS2})
$$d(z_i, w) > 3r.$$
\end{proof}


\begin{proof}[proof of Theorem \ref{weakEFq}]
Let $r  = \frac{R(1 - \aa)}{6K}$ and $m = m(\ll, \aa, K)$ be such that $S_R(x)$ contains an $m$-point $r$-net. 
By Ramsey's theorem there exists $N \in \N$ such that  every coloring of unordered  pairs of elements of $\{1,\dots,N\}$ in $m^2$ colors has a monotone $3$-element subset.

Let $x_1,\dots,x_N \subset B_r(x)$. For every $i \ne j$ we fix a minimizing geodesic connecting them.  By $\LEG$ property 
there exists at least one an extension of this geodesic beyond $y$ of length at least $\ee$. We fix one of those extensions and denote it   by $\GG_{ij}$. We denote the (first) intersection point of  $\GG_{ij}$ and $S_R(x)$ by $I_{ij}$. We say that a pair of points $\{x_i,x_j\}$
is colored in the color $\{y_k,y_l\}$ if 
$$d(y_k,I_{ij}) \le r\text{ and }d(y_l,I_{ji}) \le r,$$ 
$$\text{or}$$ 
$$d(y_l,I_{ij}) \le r\text{ and }d(y_k,I_{ji}) \le r.$$ 

By definition of $N$ there exist a monotone $3$-element set. We denote those points by $x_i, x_j, x_k$. We denote the corresponding points in the $r$-net by $y_a$ and $y_b$. Either two of $I_{ij}, I_{jk}, I_{kj}$ are in $B_r(y_a)$ or two of them are in $B_r(y_a)$. Without loss of generality can assume that $I_{ij}, I_{jk} \in B_r(y_a)$ and hence $I_{ji}, I_{kj} \in B_r(y_b)$.

Now we are going to show that 
\begin{equation}d(x_i,x_k) \ge d(x_i,x_j) + \aa d(x_k,x_j), \label{www} \end{equation}
our picture is symmetric so the proof of
$$d(x_i,x_k) \ge d(x_k,x_j) +  \aa  d(x_i,x_j)$$
is absolutely the same. 
 
By contraction let  
$$d(x_i,x_k) < d(x_i,x_j) + \aa d(x_k,x_j). $$
Let $ x'_k$ be a point on $\GG_{ij}$ such that $d(x_j,  x'_k) = d(x_j, x_k)$ and $x_i$ and $x'_k$ are on the opposite sides of $x_j$. By (\ref{www}) and the triangle inequality for $\Delta x_i x_k  x'_k $ we have,
$$d(x_k,  x'_k) > (1 - \aa) d(x_j,x_k).$$
 
By Lemma \ref{FarOnSphere}(\ref{FOS2}) the last inequality implies 
$$d(I_{ij},I_{jk}) > 3r.$$
On the other hand by $I_{ij},I_{jk} \in B_r(y_A)$ we have
$$d(I_{ij},I_{jk}) \le 2r.$$
This is a contradiction.
 \end{proof}

\section{Proofs of Theorems \ref{NoCurvesI}, \ref{Glob}, \ref{GlobQ}}
\begin{proof}[Proof of Theorem \ref{Glob}]


For every $y \in X$ there exists $N(y) \in \N$ and $r(y) > 0$ such that 
every $N(y)$-point subspace of $B_{r(y)}(y)$ does not satisfy $\SRA(\aa)$.

The family $\{B_{r(y)}(y)\}_{y \in X}$ covers  $X$, let $\{{B_{r(y_i)}}\}_{i = 1}^{k}$ be a finite subcover. Then by pigeonhole principle every $((N(y_i)+\dots+N(y_k))$-point subspace of $X$ does not satisfy $\SRA(\aa)$. 
\end{proof}

\begin{proof}[Proof of Theorem \ref{GlobQ}]
Note that since $B_R(x)$ satisfies doubling condition with constant $\ll$ we have the following.
For $m = \ll^{\left \lceil{\log_2(\frac{R}{r})}\right \rceil }$ there exist $x_1, \dots, x_m$
such that $B_R(x) \subset \cup_{k = 1}^{m}B_r(x_k)$. Thus, by pigeonhole principle
$B_R(x)$ is free of $k\ll^{\left \lceil{\log_2(\frac{R}{r})}\right \rceil}$-point $\SRA(\aa)$ subspaces.
\end{proof}

\begin{proof}[Proof of Theorem \ref{NoCurvesI}]
In the case when $X$ is an $\SRA(\aa)$-free space the result follows directly from Theorem \ref{SRAinTC}.
Thus, it remains to consider the case when $X$ is a proper locally $\SRA(\aa)$-free geodesic metric space.

By contradiction suppose that there exists a point $x \in X$, $R > 0$ and a non rectifiable self-contracted curve $\gg:I \rightarrow B_R(x)$. Since $X$ is proper we have that  $\overline{B_{R}(x)}$ is compact.  By Theorem \ref{Glob} there exists $N \in \N$, such that $B_R(x)$ is free of $N$-point $\SRA(\aa)$ subspaces.  On the other hand by Theorem \ref{SRAinTC} $\gg(I)$ should contain an $N$-point $\SRA(\aa)$ subspace. Thus we have a contradiction.
\end{proof}


\subsection*{Acknowledgements}
I thank Sergey V. Ivanov for advising me during this work and especially for several key hints for the proof of Theorem \ref{SRAinTC}. 
I'm grateful to  Alexander Lytchak  for the idea of considering spaces with extendable geodesics.
 I thank Andrey Alpeev who introduced me into results of E. Stepanov and Y. Teplitskaya and the field of self-contracted curves in general.
I'm grateful to Nina Lebedeva and Tapio Rajala for fruitful discussions. 
Part of the work was done during the intense activity period ``Metric Measure Spaces and Ricci Curvature". I thank Matthias Erbar and Karl-Theodor Sturm for organizing this event. The paper is supported by DFG grant SPP 2026.

\bibliography{circle}

\bibliographystyle{plain}

\end{document}